\documentclass[12pt]{amsart}
\usepackage{amscd,amssymb,amsmath,multicol}
\newtheorem{thm}[equation]{Theorem}
\numberwithin{equation}{section}

\newtheorem{conj}[equation]{Conjecture}
\newtheorem{defin}[equation]{Definition}

\newtheorem{prop}[equation]{Proposition}

\begin{document}
\raggedbottom \voffset=-.7truein \hoffset=0truein \vsize=8truein
\hsize=6truein \textheight=8truein \textwidth=6truein
\baselineskip=18truept
\def\mapright#1{\ \smash{\mathop{\longrightarrow}\limits^{#1}}\ }
\def\ss{\smallskip}
\def\ssum{\sum\limits}
\def\dsum{{\displaystyle{\sum}}}
\def\la{\langle}
\def\ra{\rangle}
\def\on{\operatorname}
\def\o{\on{od}}
\def\lg{\on{lg}}
\def\a{\alpha}
\def\bz{{\Bbb Z}}
\def\eps{\epsilon}
\def\br{{\bold R}}
\def\bc{{\bold C}}
\def\bN{{\bold N}}
\def\nut{\widetilde{\nu}}
\def\tfrac{\textstyle\frac}
\def\product{\prod}
\def\b{\beta}
\def\G{\Gamma}
\def\g{\gamma}
\def\zt{{\Bbb Z}_2}
\def\zth{{\bold Z}_2^\wedge}
\def\bs{{\bold s}}
\def\bg{{\bold g}}
\def\bof{{\bold f}}
\def\bq{{\bold Q}}
\def\be{{\bold e}}
\def\line{\rule{.6in}{.6pt}}
\def\xb{{\overline x}}
\def\xbar{{\overline x}}
\def\ybar{{\overline y}}
\def\zbar{{\overline z}}
\def\ebar{{\overline \be}}
\def\nbar{{\overline n}}
\def\fbar{{\overline f}}
\def\Ubar{{\overline U}}
\def\et{{\widetilde e}}
\def\ni{\noindent}
\def\ms{\medskip}
\def\ahat{{\hat a}}
\def\bhat{{\hat b}}
\def\chat{{\hat c}}
\def\nbar{{\overline{n}}}
\def\minp{\min\nolimits'}
\def\N{{\Bbb N}}
\def\Z{{\Bbb Z}}
\def\Q{{\Bbb Q}}
\def\R{{\Bbb R}}
\def\C{{\Bbb C}}
\def\el{\ell}
\def\mo{\on{mod}}
\def\dstyle{\displaystyle}
\def\ds{\dstyle}
\def\Remark{\noindent{\it  Remark}}
\title
{For which $2$-adic integers $x$ can $\dstyle\sum_k\tbinom xk^{-1}$ be defined?}
\author{Donald M. Davis}
\address{Department of Mathematics, Lehigh University\\Bethlehem, PA 18015, USA}
\email{dmd1@lehigh.edu}
\date{January 11, 2013}


\maketitle
\begin{abstract} Let $f(n)=\sum_k\binom nk^{-1}$. In a previous paper, we defined for a $p$-adic integer $x$ that $f(x)$ is $p$-{\it definable} if $\lim f(x_j)$ exists in $\Q_p$,
where $x_j$ denotes the mod $p^j$ reduction of $x$. We proved that if $p$ is  odd, then
$-1$ is the only element of $\Z_p-\N$ for which $f(x)$ is $p$-definable. For $p=2$, we
proved that if the 1's in the binary expansion of $x$ are eventually extraordinarily sparse, then $f(x)$ is 2-definable.
Here we present some conjectures that $f(x)$ is 2-definable for many more 2-adic integers. We discuss the extent to which we can prove these conjectures.
\end{abstract}
\section{Statement of conjectures and their consequences}\label{intro}
Let $\N\subset\Z_p\subset\Q_p$ denote the natural numbers (including 0), $p$-adic integers, and $p$-adic numbers, respectively, with metric $d_p(x,y)=p^{-\nu_p(x-y)}$.
Here and throughout, $\nu_p(-)$ denotes the exponent of $p$ in a rational number.
Let $f:\N\to\Q_p$ be defined by
$$f(n)=\sum_{k=0}^n\tbinom nk^{-1}.$$
In \cite{D}, we made the following definition.
\begin{defin} Let $x\in\Z_p$, and let $x_j$ denote the mod $p^{j}$ reduction of $x$. Then $f(x)$ is $p$-definable if $\la f(x_j)\ra$ is a Cauchy sequence in $\Q_p$.\end{defin}
\noindent Then $f(x)$ could be defined to be the limit in $\Q_p$ of this Cauchy sequence.

We proved in \cite{D} that if $p$ is an odd prime, then $f(x)$ is $p$-definable if and only if $x=-1$ or $x\in\N$. (Actually, $p$ was required to satisfy a technical condition
which is satisfied by all primes less than $10^8$, and for which there are no primes which are known not to satisfy it.) We also proved that if $x=\sum 2^{e_i}$ with $e_i<e_{i+1}$, then $f(x)$  is 2-definable
if, roughly, $i+1>2^i$  for all sufficiently large $i$. The 1's in the binary expansion of such an $x$ are eventually extraordinarily sparse.
Here we discuss our attempts to prove that $f(x)$ is 2-definable for many more 2-adic integers.

Let $\a(n)$ denote the number of 1's in the binary expansion of $n$, $\lg(-)=[\log_2(-)]$, and $\nu(-)=\nu_2(-)$.
Our strongest conjecture is
\begin{conj}\label{conj1} If $0\le k<2^e$, then
$$\nu(f(2^e+k)-f(k))\ge e-2\a(k)-2.$$\end{conj}

Conjecture \ref{conj1} has been verified for $e\le15$. In this range, equality holds iff $k=2^e-4$ or $2^e-2$.
The following result describes the consequence of this conjecture for 2-definability.
\begin{prop}\label{cor1} Assume Conjecture \ref{conj1}. If the number of 0's minus the number of 1's in $x_j$ approaches $\infty$ as $j$ goes to $\infty$, then $f(x)$ is 2-definable.\end{prop}
\noindent We include  leading 0's in $x_j$ here, since they will eventually be seen. An alternative statement is that $f(x)$ would be 2-definable if the fraction of 0's in $x$ is greater than 1/2.
\begin{proof}[Proof of Proposition \ref{cor1}] Let $\ds x=\sum_{i=1}^\infty 2^{e_i}$ with $e_i<e_{i+1}$. The $i$th distinct point in the sequence of $f(x_j)$'s is $f(2^{e_i}+x_{e_i})$,
and the $(i-1)$st distinct point is $f(x_{e_i})$. The distance between these points is $2^{-v}$, where $$v=\nu(f(2^{e_i}+x_{e_i})-f(x_{e_i}))\ge e_i-2\a(x_{e_i})-2,$$
according to Conjecture \ref{conj1}. The number of 0's in $x_{e_i}$ equals $e_i-\a(x_{e_i})$.
Our hypothesis  says that $e_i-2\a(x_{e_i})$ becomes arbitrarily large, and hence the distance between the $i$th and $(i-1)$st distinct points in the sequence
is $2^{-v}$ where $v$ becomes arbitrarily large. Thus our sequence is Cauchy.
\end{proof}

Although we have very strong evidence for Conjecture \ref{conj1}, we feel that we are more likely to be able to prove the following conjecture.
  \begin{conj}\label{conj2} If $0\le k<2^{e-1}$,  then
$$\nu(f(2^e+2k+1)-f(2k+1))\ge e-2\lg(k+3)+2\nu(k+1).$$
\end{conj}

Conjecture \ref{conj2} has been verified for $e\le15$. In this range, equality holds iff $k=2^{e-1}-2$.
The following result describes the consequence of this conjecture for 2-definability.
\begin{prop}\label{cor2} Assume Conjecture \ref{conj2}. Suppose $x=\sum 2^{e_i}$ has $e_1=0$ and $e_i<e_{i+1}$ and  satisfies $\ds\lim_{i\to\infty}(e_{i+1}-2e_i)=\infty$. Then $f(x)$ is 2-definable.\end{prop}
Note that this would be exponentially stronger than the result proved in \cite{D} and referenced above, but still much weaker than the conclusion of Proposition \ref{cor1}.
\begin{proof}[Proof of Proposition \ref{cor2}] Arguing similarly to the previous proof, the distance between consecutive points in the sequence is $2^{-v}$ with
 $$v=\nu(f(2^{e_i}+x_{e_i})-f(x_{e_i}))\ge e_i-2\lg(x_{e_i}+3)\ge e_i-2e_{i-1}-2$$
 according to Conjecture \ref{conj2}. Since our assumption is that $v$ becomes arbitrarily large, the sequence is Cauchy.
  \end{proof}

\section{Steps toward a proof of Conjecture \ref{conj2}}\label{pfsec}
In this section, we outline a program which we hope might lead to a proof of Conjecture \ref{conj2}.
Using symmetry of binomial coefficients, the following result is immediate.
\begin{prop} \label{symm}Let $0\le k<2^{e-1}$. If the following two statements are true, then so is Conjecture \ref{conj2}.
\begin{enumerate}\item[i.] ${\ds\nu\bigl(\sum_{i=0}^k} \bigl(\binom{2^e+2k+1}i^{-1}-\binom{2k+1}i^{-1}\bigr)\bigr)\ge e-2\lg(k+2)+2\nu(k+1)$,
\item[ii.] ${\ds\nu\bigl(\sum_{i=k+1}^{2^{e-1}+k}}\binom{2^e+2k+1}i^{-1}\bigr)\ge e-2\lg(k+3)+2\nu(k+1)-1.$
\end{enumerate}
\end{prop}

Our main result is
\begin{thm}\label{thm} Let $0\le k<2^{e-1}$. Then statement i.~of Proposition \ref{symm} is true. Indeed, with
$$T_i:=\tbinom{2^e+2k+1}i^{-1}-\tbinom{2k+1}i^{-1},$$
we have
\begin{enumerate}
\item[a.] if $0\le i\le [(k-1)/2]$, then
$$\nu(T_{2i}+T_{2i+1})\ge e-2\lg(k+1)+2\nu(k+1),\quad{\text and}$$
\item[b.] if $k$ is even, then
$$\nu(T_k)\ge e-2\lg(k+2).$$
\end{enumerate}
\end{thm}

Our proof will  use  the standard results that $\nu\binom{m+n}m=\a(m)+\a(n)-\a(m+n)$, and that $\nu\binom{m+n}m$ equals the number of carries when
$m$ and $n$ are added in binary arithmetic. It follows from this that
\begin{equation}\nu\tbinom ki\le\lg(k+1)-\nu(k+1),\label{carries}\end{equation}
since, if $\nu(k+1)=t$, then there cannot be any carries in the last $t$ positions in the binary addition of $i$ and $k-i$.

\begin{proof}[Proof of part b of Theorem \ref{thm}] We first note that
\begin{equation}\label{sumj}\tbinom{2^e+a}b^{-1}-\tbinom ab^{-1}=-\tbinom{2^e+a}b^{-1}\sum_{j\ge1}2^{je}\sigma_j(\tfrac1a,\ldots,\tfrac1{a-b+1}),\end{equation}
where $\sigma_j(-)$ denotes an elementary symmetric function.

Let $k=2\ell$.
Including only the $(j=1)$-term, which we will justify, (\ref{sumj}) yields that $T_{2\ell}$ has the same 2-exponent as
\begin{equation}\label{expr}2^e\tbinom{2^e+4\ell+1}{2\ell}^{-1}\bigl(\tfrac1{2\ell+2}+\cdots+\tfrac1{4\ell+1}\bigr).\end{equation}
Note that $2\ell+2\le2^t\le4\ell+1$ iff $2^{t-2}\le\ell\le2^{t-1}-1$, and so $\nu(\frac1{2\ell+2}+\cdots+\frac1{4\ell+1})=-\lg(\ell)-2$.
Thus the 2-exponent of (\ref{expr}) equals $e-\a(\ell)-\lg(\ell)-2\ge e-2\lg(2\ell+2)$,  as claimed. Here we use that  $2\lg(\ell+1)\ge\a(\ell)+\lg(\ell)$, which is proved by considering separately
$2^t\le \ell<2^{t+1}-1$ and $\ell=2^{t+1}-1$.

Now we justify including only the term with $j=1$ in the above sum. Let
$$v_j=\nu(2^{je}\sigma_j(\tfrac1{2\ell+2},\ldots,\tfrac1{4\ell+1})).$$
If $\nu(\sigma_1(-))=-t$, then $v_1=e-t>0$, and if $j>1$ then $v_j> j(e-t)>v_1$,
since $\sigma_j(-)$ is a sum of products of $j$ factors, each with 2-exponent $\ge-t$, and at most one equal to $-t$.
\end{proof}
\begin{proof}[Proof of part a of Theorem \ref{thm}]
Including only the $(j=1)$-term of (\ref{sumj}), which again will be justified, we obtain that $T_{2i}+T_{2i+1}$ equals
\begin{equation}\label{ii+1}-2^e\tbinom{2^e+2k+1}{2i}^{-1}\bigl(\bigl(\tfrac1{2k+1}+\cdots+\tfrac1{2k-2i+2}\bigr)\bigr(1+\tfrac{2i+1}{2^e+2k-2i+1}\bigr)
+\tfrac{2i+1}{(2^e+2k-2i+1)(2k-2i+1)}\bigr).\end{equation}
Thus, using (\ref{carries}) at the second step,
\begin{eqnarray*}\nu(T_{2i}+T_{2i+1})&\ge& e-\nu\tbinom ki+\min(-\lg(2k)+\nu(2^e+2k+2),0)\\
&\ge&\min(e+2\nu(k+1)-\lg(k+1)-\lg(k),e-\lg(k+1)+\nu(k+1)),\end{eqnarray*}
which is as claimed.

We complete the proof by showing that if $j>1$, then using the $j$-term of the sum in (\ref{sumj}) in $T_{2i}+T_{2i+1}$ would give an expression with  2-exponent
at least as large as was obtained with $j=1$. Analogous to part of (\ref{ii+1}), the $j$-term would be, up to odd multiples, 
\begin{equation}2^{je}((2^e+2k+2)\sigma_j(-)+\sigma_{j-1}(-)).\label{j}\end{equation}
If $\nu(\sigma_1(-))=-t$, then $\nu(\sigma_j(-))>-jt$. Since $e>t$ and $e>\nu(2k+2)$, the claim when $k<2^{e-1}-1$ follows from
$$je+\nu(2k+2)-jt>e+\nu(2k+2)-t$$
and $$je-(j-1)t> e+\nu(2k+2)-t.$$
If $k=2^{e-1}-1$, then $t=e-1$ and (\ref{j}) has 2-exponent $e$ if $j=1$ (from $\sigma_0(-)$) and a larger value if $j>1$.
\end{proof}

Despite much effort, we have been unable to prove statement ii.~of Proposition \ref{symm}. Note that the application to 2-definability given in Proposition \ref{cor2}
would be true even if  Conjecture \ref{conj2} or Proposition \ref{symm} did not contain the ``$+2\nu(k+1)$.'' 

\def\line{\rule{.6in}{.6pt}}

\end{document}